\documentclass[proceedings,submission]{dmtcs}

\usepackage[latin1]{inputenc}
\usepackage{indentfirst}
\usepackage{amsmath,amsfonts,amssymb,amsopn}
\usepackage{mathrsfs}

\author{Pierre-Lo\"ic M\'eliot\addressmark{1}}
\title[Products of Geck-Rouquier conjugacy classes]{Products of Geck-Rouquier conjugacy classes and the Hecke algebra of composed permutations}
\address{\addressmark{1}Institut Gaspard Monge --- Universit\'e Paris-Est Marne-La-Vall\'ee --- 77454 Marne-La-Vall\'ee cedex 2 \\
meliot@phare.normalesup.org}
\keywords{Iwahori-Hecke algebras, Geck-Rouquier conjugacy classes, symmetric functions.}

\begin{document}
\maketitle
\begin{abstract}
\paragraph{Abstract.}
We show the $q$-analog of a well-known result of Farahat and Higman: in the center of the Iwahori-Hecke algebra $\mathscr{H}_{n,q}$, if $(a_{\lambda\mu}^{\nu}(n,q))_{\nu}$ is the set of structure constants involved in the product of two Geck-Rouquier conjugacy classes $\Gamma_{\lambda, n}$ and $\Gamma_{\mu,n}$, then each coefficient $a_{\lambda\mu}^{\nu}(n,q)$ depend on $n$ and $q$ in a polynomial way. Our proof relies on the construction of a projective limit of the Hecke algebras; this projective limit is inspired by the Ivanov-Kerov algebra of partial permutations.

\paragraph{R\'esum\'e.}
Nous d\'emontrons le $q$-analogue d'un r\'esultat bien connu de Farahat et Higman : dans le centre de l'alg\`ebre d'Iwahori-Hecke $\mathscr{H}_{n,q}$, si $(a_{\lambda\mu}^{\nu}(n,q))_{\nu}$ est l'ensemble des constantes de structure mises en jeu dans le produit de deux classes de conjugaison de Geck-Rouquier $\Gamma_{\lambda,n}$ et $\Gamma_{\mu,n}$, alors chaque coefficient $a_{\lambda\mu}^{\nu}(n,q)$ d\'epend de fa\c con polynomiale de $n$ et de $q$. Notre preuve repose sur la construction d'une limite projective des alg\`ebres d'Hecke ; cette limite projective est inspir\'ee de l'alg\`ebre d'Ivanov-Kerov des permutations partielles.
\end{abstract}

\newcommand{\Z}{\integers}                
\newcommand{\N}{\naturals}                
\newcommand{\Q}{\rationals}
\newcommand{\C}{\complexes}            
\newcommand{\IH}{\mathscr{H}}          
\newcommand{\sym}{\mathfrak{S}}      
\newcommand{\Comp}{\mathfrak{C}}   
\newcommand{\Part}{\mathfrak{P}}      
\newcommand{\Zen}{\mathscr{Z}}        
\newcommand{\pr}{\mathrm{pr}}           
\newcommand{\alg}{\mathscr{A}}
\newcommand{\blg}{\mathscr{B}}
\newcommand{\clg}{\mathscr{C}}
\newcommand{\dlg}{\mathscr{D}}
\newcommand{\lle}{\left[\!\left[}              
\newcommand{\rre}{\right]\!\right]}       
\newtheorem{theorem}{Theorem}
\newtheorem{proposition}[theorem]{Proposition}

\bigskip

In this paper, we answer a question asked in \cite{FW09} that concerns the products of Geck-Rouquier conjugacy classes in the Hecke algebras $\IH_{n,q}$. If $\lambda=(\lambda_{1},\lambda_{2},\ldots,\lambda_{r})$ is a partition with $|\lambda| +\ell(\lambda)\leq n$, we consider the completed partition
$$\lambda\rightarrow n=(\lambda_{1}+1,\lambda_{2}+1,\ldots,\lambda_{r}+1,1^{n-|\lambda|-\ell(\lambda)}),$$
and we denote by $C_{\lambda,n}=C_{\lambda \rightarrow n}$ the corresponding conjugacy class, that is to say, the sum of all permutations with cycle type $\lambda \rightarrow n$ in the   center of the symmetric group algebra $\C\sym_{n}$. Notice that in particular, $C_{\lambda,n}=0$ if $|\lambda|+\ell(\lambda)> n$. It is known since \cite{FH59} that the products of completed conjugacy classes write as
$$C_{\lambda,n}\,*\,C_{\mu,n}=\sum_{|\nu| \leq |\lambda|+|\mu|}\,a_{\lambda\mu}^{\nu}(n)\,C_{\nu, n},$$
where the structure constants $a_{\lambda\mu}^{\nu}(n)$ depend on $n$ in a polynomial way. In \cite{GR97}, some deformations $\Gamma_{\lambda}$ of the conjugacy classes $C_{\lambda}$ are constructed. These central elements form a basis of the center $\Zen_{n,q}$ of the Iwahori-Hecke algebra $\IH_{n,q}$, and they are characterized by the two following properties, see \cite{Fra99}:
\begin{enumerate}
\item The element $\Gamma_{\lambda}$ is central and specializes to $C_{\lambda}$ for $q=1$.\vspace{-2mm}
\item The difference $\Gamma_{\lambda}-C_{\lambda}$ involves no permutation of minimal length in its conjugacy class.
\end{enumerate}
As before, $\Gamma_{\lambda,n}=\Gamma_{\lambda \rightarrow n}$ if $|\lambda|+\ell(\lambda) \leq n$, and $0$ otherwise. Our main result is the following:

\begin{theorem}\label{holdupmental}
In the center of the Hecke algebra $\IH_{n,q}$, the products of completed Geck-Rouquier conjugacy classes write as
$$\Gamma_{\lambda,n}\,*\,\Gamma_{\mu,n}=\sum_{|\nu| \leq |\lambda|+|\mu|}\,a_{\lambda\mu}^{\nu}(n,q)\,\,\Gamma_{\nu,n},$$
and the structure constants $a_{\lambda\mu}^{\nu}(n,q)$ are in $\Q[n,q,q^{-1}]$.
\end{theorem}\bigskip

The first part of Theorem \ref{holdupmental} --- that is to say, that elements $\Gamma_{\nu,n}$ involved in the product satisfy the inequality $|\nu| \leq |\lambda|+|\mu|$ --- was already in \cite[Theorem 1.1]{FW09}, and the polynomial dependance of the coefficients $a_{\lambda\mu}^{\nu}(n,q)$ was Conjecture 3.1;  our paper is devoted to a proof of this conjecture. We shall combine two main arguments: 
\begin{itemize}
\item We construct a projective limit $\dlg_{\infty,q}$ of the Hecke algebras, which is essentially a $q$-version of the algebra of Ivanov and Kerov, see \cite{IK99}. We perform \emph{generic computations} inside various subalgebras of $\dlg_{\infty,q}$, and we project then these calculations on the algebras $\IH_{n,q}$ and their centers. \vspace{-2mm}
\item The centers of the Hecke algebras admit numerous bases, and these bases are related one to another in the same way as the bases of the symmetric function algebra $\Lambda$. This allows to separate the dependance on $q$ and the dependance on $n$ of the coefficients $a_{\lambda\mu}^{\nu}(n,q)$.
\end{itemize}
Before we start, let us fix some notations. If $n$ is a non-negative integer, $\Part_{n}$ is the set of partitions of $n$, $\Comp_{n}$ is the set of compositions of $n$, and $\sym_{n}$ is the set of permutations of the interval $\lle 1,n\rre$. The \textbf{type} of a permutation $\sigma \in \sym_{n}$ is the partition $\lambda=t(\sigma)$ obtained by ordering the sizes of the orbits of $\sigma$; for instance, $t(24513)=(3,2)$. The \textbf{code} of a composition $c\in \Comp_{n}$ is the complementary in $\lle 1,n \rre$ of the set of descents of $c$; for instance, the code of $(3,2,3)$ is $\{1,2,4,6,7\}$. Finally, we denote by $\Zen_{n}=Z(\C\sym_{n})$ the center of the algebra $\C\sym_{n}$; the conjugacy classes $C_{\lambda}$ form a linear basis of $\Zen_{n}$  when $\lambda$ runs over $\Part_{n}$.\bigskip

\section{Partial permutations and the Ivanov-Kerov algebra}
Since our argument is essentially inspired by the construction of \cite{IK99}, let us recall it briefly. A \textbf{partial permutation} of order $n$ is a pair $(\sigma,S)$ where $S$ is a subset of $\lle 1,n\rre$, and $\sigma$ is a permutation in $\sym(S)$. Alternatively, one may see a partial permutation as a permutation $\sigma$ in $\sym_{n}$ together with a subset containing the non-trivial orbits of $\sigma$. The product of two partial permutations is
$$(\sigma,S) \,(\tau,T)=(\sigma\tau,S\cup T),$$
and this operation yield a semigroup whose complex algebra is denoted by $\blg_{n}$. There is a natural projection $\pr_{n} : \blg_{n}\to \C\sym_{n}$ that consists in forgetting the support of a partial permutation, and also natural compatible maps
$$\phi_{N,n} :(\sigma,S) \in \blg_{N}\mapsto
 \begin{cases}(\sigma,S) \in \blg_{n}&\text{if }S\subset\lle 1,n\rre,
\\0&\text{otherwise,}
\end{cases}$$
whence a projective limit $\blg_{\infty}=\varprojlim \blg_{n}$ with respect to this system $(\phi_{N,n})_{N\geq n}$ and in the category of filtered algebras. Now, one can lift the conjugacy classes $C_{\lambda}$ to the algebras of partial permutations. Indeed, the symmetric group $\sym_{n}$ acts on $\blg_{n}$ by 
$$\sigma\cdot(\tau,S)=(\sigma\tau\sigma^{-1},\sigma(S)),$$
and a linear basis of the invariant subalgebra $\alg_{n}=(\blg_{n})^{\sym_{n}}$ is labelled by the partitions $\lambda$ of size less than or equal to $n$:
$$\alg_{n}=\bigoplus_{|\lambda| \leq n}\C A_{\lambda,n},\qquad\text{where }A_{\lambda,n}=\!\!\!\!\sum_{\substack{|S|=|\lambda|\\
\sigma \in \sym(S),\,\,t(\sigma)=\lambda}} \!\!\!\!(\sigma,S).$$
Since the actions $\sym_{n} \curvearrowright \blg_{n}$ are compatible with the morphisms $\phi_{N,n}$,  the inverse limit $\alg_{\infty}=(\blg_{\infty})^{\sym_{\infty}}$ of the invariant subalgebras has a basis $(A_{\lambda})_{\lambda}$ indexed by all partitions $\lambda \in \Part=\bigsqcup_{n \in \N} \,\Part_{n}$, and such that $\phi_{\infty,n}(A_{\lambda})=A_{\lambda,n}$ (with by convention $A_{\lambda,n}=0$ if $|\lambda|>n$). As a consequence, if $(a_{\lambda\mu}^{\nu})_{\lambda,\mu,\nu}$ is the family of structure constants of the \textbf{Ivanov-Kerov algebra}\footnote{It can be shown that $\alg_{\infty}$ is isomorphic to the algebra of shifted symmetric polynomials, see Theorem 9.1 in \cite{IK99}.} $\alg_{\infty}$ in the basis $(A_{\lambda})_{\lambda \in \Part}$, then
$$\forall n,\,\,\,A_{\lambda,n}\,*\,A_{\mu,n}=\sum_{\nu} \,a_{\lambda\mu}^{\nu}\,A_{\nu,n},$$
with $A_{\lambda,n}=0$ if $|\lambda|\geq n$. Moreover, it is not difficult to see that $a_{\lambda\mu}^{\nu}\neq 0$ implies $|\nu|\leq |\lambda|+|\mu|$, and also $|\nu|-\ell(\nu)\leq |\lambda|-\ell(\lambda)+|\mu|-\ell(\mu)$, \emph{cf.} \cite[\S10]{IK99}, for the study of the filtrations of $\alg_{\infty}$. Now, $\pr_{n}(\alg_{n})=\Zen_{n}$, and more precisely, $$\pr_{n}(A_{\lambda,n})=\binom{n-|\lambda|+m_{1}(\lambda)}{m_{1}(\lambda)}\,C_{\lambda-1,n}.$$
where $\lambda-1=(\lambda_{1}-1,\ldots,\lambda_{s}-1)$ if $\lambda=(\lambda_{1},\ldots,\lambda_{s}\geq 2,1,\ldots,1)$. The result of Farahat and Higman follows immediately, and we shall try to mimic this construction in the context of Iwahori-Hecke algebras.\bigskip

\section{Composed permutations and their Hecke algebra}
We recall that the \textbf{Iwahori-Hecke algebra} of type A and order $n$ is the quantized version of the symmetric group algebra defined over $\C(q)$ by
$$\IH_{n,q}=\left\langle S_{1},\ldots,S_{n-1}\,\,\bigg|\,\,\substack{\,\, \text{braid relations: }\forall i,\,\,S_{i}S_{i+1}S_{i}=S_{i+1}S_{i}S_{i+1}\quad\\
\text{commutation relations: }\forall |j-i|>1,\,\,S_{i}S_{j}=S_{j}S_{i} \\
\text{quadratic relations: }\forall i,\,\,(S_{i})^{2}=(q-1)\,S_{i}+q\quad\,
 }\right\rangle.$$
When $q=1$, we recover the symmetric group algebra $\C\sym_{n}$. If $\omega \in \sym_{n}$, let us denote by $T_{\omega}$ the product $S_{i_{1}}S_{i_{2}}\cdots S_{i_{r}}$, where $\omega=s_{i_{1}}s_{i_{2}}\cdots s_{i_{r}}$ is any reduced expression of $\omega$ in elementary transpositions $s_{i}=(i,i+1)$. Then, it is well known that the elements $T_{\omega}$ do not depend on the choice of reduced expressions, and that they form a $\C(q)$-linear basis of $\IH_{n,q}$, see \cite{Mat99}.\bigskip

In order to construct a \emph{projective} limit of the algebras $\IH_{n,q}$, it is very tempting to mimic the construction of Ivanov and Kerov, and therefore to build an Hecke algebra of partial permutations. Unfortunately, this is not possible; let us explain why by considering for instance the transposition $\sigma=1432$ in $\sym_{4}$. The possible supports for $\sigma$ are $\{2,4\}$, $\{1,2,4\}$, $\{2,3,4\}$ and $\{1,2,3,4\}$. However,
$$\sigma=s_{2}s_{3}s_{2},$$
and the support of $s_{2}$ (respectively, of $s_{3}$) contains at least $\{2,3\}$ (resp., $\{3,4\}$). So, if we take account of the Coxeter structure of $\sym_{4}$ --- and it should obviously be the case in the context of Hecke algebras --- then the only valid supports for $\sigma$ are the connected ones, namely, $\{2,3,4\}$ and $\{1,2,3,4\}$. This problem leads to consider \emph{composed permutations} instead of \emph{partial permutations}. If $c$ is a composition of $n$, let us denote by $\pi(c)$ the corresponding set partition of $\lle 1,n\rre$, \emph{i.e.}, the set partition whose parts are the intervals $\lle 1,c_{1}\rre,\,\,\lle c_{1}+1,c_{1}+c_{2}\rre$, etc. A \textbf{composed permutation} of order $n$ is a pair $(\sigma,c)$ with $\sigma \in \sym_{n}$ and $c$ composition in $\Comp_{n}$ such that $\pi(c)$ is coarser than the set partition of orbits of $\sigma$. For instance, $(32154867,(5,3))$ is a composed permutation of order $8$; we shall also write this $32154|867$. The product of two composed permutations is then defined by
$$(\sigma,c)\,(\tau,d)=(\sigma\tau,c\vee d),$$
where $c\vee d$ is the finest composition of $n$ such that $\pi(c \vee d)\geq \pi(c)\vee \pi(d)$ in the lattice of set partitions. For instance,
$$321|54|867 \,\times\, 12|435|687 = 42153|768.$$
One obtains so a semigroup of composed permutations; its complex semigroup algebra will be denoted by\footnote{If one considers pairs $(\sigma,\pi)$ where $\pi$ is any set partition of $\lle 1,n\rre$ coarser than $\mathrm{orb}(\sigma)$ (and not necessarily a set partition associated to a composition), then one obtains an algebra of \emph{split permutations} whose subalgebra of invariants is related to the connected Hurwitz numbers $H_{n,g}(\lambda)$.} $\dlg_{n}$, and the dimension of $\dlg_{n}$ is the number of composed permutations of order $n$.
\bigskip

Now, let us describe an Hecke version $\dlg_{n,q}$ of the algebra $\dlg_{n}$. As for $\IH_{n,q}$, one introduces generators $(S_{i})_{1\leq i \leq n-1}$ corresponding to the elementary transpositions $s_{i}$, but one has also to introduce generators $(I_{i})_{1\leq i \leq n-1}$ that allow to join the parts of the composition of a composed permutation. Hence, the \textbf{Iwahori-Hecke algebra of composed permutations} is defined (over the ground field $\C(q)$) by $\dlg_{n,q}=\langle S_{1},\ldots,S_{n-1},I_{1},\ldots,I_{n-1}\rangle$ and the following relations:
\begin{align*}
&\forall i,\,\,\,S_{i}S_{i+1}S_{i}=S_{i+1}S_{i}S_{i+1}\\
&\forall |j-i|>1,\,\,\,S_{i}S_{j}=S_{j}S_{i}\\
&\forall i,\,\,\,(S_{i})^{2}=(q-1)\,S_{i}+q\,I_{i}\\
&\forall i,j,\,\,\,S_{i}I_{j}=I_{j}S_{i} \\
&\forall i,j,\,\,\,I_{i}I_{j}=I_{j}I_{i}\\
&\forall i,\,\,\,S_{i}I_{i}=S_{i}\\
&\forall i,\,\,\,(I_{i})^{2}=I_{i}
\end{align*}
The generators $S_{i}$ correspond to the composed permutations $1|2|\ldots|i-1|i+1,i|i+2|\ldots|n$, and the generators $I_{i}$ correspond to the composed permutations $1|2|\ldots|i-1|i,i+1|i+2|\ldots|n$.

\begin{proposition}\label{psychoactive}
The algebra $\dlg_{n,q}$ specializes to the algebra of composed permutations $\dlg_{n}$ when $q=1$; to the Iwahori-Hecke algebra $\IH_{n,q}$ when $I_{1}=I_{2}=\cdots=I_{n-1}=1$; and to the algebra $\dlg_{m,q}$ of lower order $m<n$ when $I_{m}=I_{m+1}=\cdots=I_{n-1}=0$ and $S_{m}=S_{m+1}=\cdots=S_{n-1}=0$.
\end{proposition}
\bigskip

In the following, we shall denote by $\pr_{n}$ the specialization $\dlg_{n,q} \to \IH_{n,q}$; it generalizes the map $\dlg_{n}\to \C\sym_{n}$ of the first section. The first part of Proposition \ref{psychoactive} is actually the only one that is non trivial, and it will be a consequence of Theorem \ref{amnesia}. If $\omega$ is a permutation with reduced expression $\omega=s_{i_{1}}s_{i_{2}}\cdots s_{i_{r}}$, we denote as before by $T_{\omega}$ the product $S_{i_{1}}S_{i_{2}}\ldots S_{i_{r}}$ in $\dlg_{n,q}$. On the other hand, if $c$ is a composition of $\lle 1,n\rre$,  we denote by $I_{c}$ the product of the generators $I_{j}$ with $j$ in the code of $c$ (so for instance, $I_{(3,2,3)}=I_{1}I_{2}I_{4}I_{6}I_{7}$ in $\dlg_{8,q}$). These elements are central idempotents, and $I_{c}$ correspond to the composed permutation $(\mathrm{id},c)$. Finally, if $(\sigma,c)$ is a composed permutation, $T_{\sigma,c}$ is the product $T_{\sigma}I_{c}$.

\begin{theorem}\label{amnesia}
In $\dlg_{n,q}$, the products $T_{\sigma}$ do not depend on the choice of reduced expressions, and the products $T_{\sigma,c}$ form a linear basis of $\dlg_{n,q}$ when $(\sigma,c)$ runs over composed permutations of order $n$. There is an isomorphism of $\C(q)$-algebras between 
$$\dlg_{n,q} \quad\text{and}\quad\bigoplus_{c \in \Comp_{n}} \IH_{c,q},$$
where $\IH_{c,q}$ is the Young subalgebra $\IH_{c_{1},q}\otimes \IH_{c_{2},q}\otimes\cdots \otimes \IH_{c_{r},q}$ of $\IH_{n,q}$.
\end{theorem}

\begin{proof}
If $\sigma\in \sym_{n}$, the Matsumoto theorem ensures that it is always possible to go from a reduced expression $s_{i_{1}}s_{i_{2}}\cdots s_{i_{r}}$ to another reduced expression $s_{j_{1}}s_{j_{2}}\cdots s_{j_{r}}$ by braid moves $s_{i}s_{i+1}s_{i} \leftrightarrow s_{i+1}s_{i}s_{i+1}$
and commutations $s_{i}s_{j}  \leftrightarrow s_{j}s_{i}$ when $|j-i|>1$. Since the corresponding products of $S_{i}$ in $\dlg_{n,q}$ are preserved by these substitutions, a product $T_{\sigma}$ in $\dlg_{n,q}$ does not depend on the choice of a reduced expression. Now, let us consider an arbitrary product $\Pi$ of generators $S_{i}$ and $I_{j}$ (in any order). As the elements $I_{j}$ are central idempotents, it is always possible to reduce the product to
$$\Pi=S_{i_{1}}S_{i_{2}}\cdots S_{i_{p}}\,I_{c}$$
with $c$ composition of $n$ --- here, $s_{i_{1}}s_{i_{2}}\cdots s_{i_{p}}$ is \emph{a priori} not a reduced expression. Moreover, since $S_{i}\,I_{i}=S_{i}$, we can suppose that the code of $c$ contains $\{i_{1},\ldots,i_{p}\}$. Now, suppose that $\sigma=s_{i_{1}}s_{i_{2}}\cdots s_{i_{p}}$ is not a reduced expression. Then, by using braid moves and commutations, we can transform the expression in one with two consecutive letters that are identical, that is to say that if $j_k = j_{k+1}$,
$$\sigma=s_{j_{1}}\cdots s_{j_{k}}s_{j_{k+1}} \cdots s_{j_{p}}=s_{j_{1}}\cdots s_{j_{k-1}}s_{j_{k+2}}\cdots s_{j_{p}}.$$
We apply the same moves to the $S_{i}$ in $\dlg_{n,q}$ and we obtain $\Pi=S_{j_{1}}\cdots S_{j_{k}}S_{j_{k+1}}\cdots S_{j_{p}} \,I_{c}$; 
notice that the code of $c$ still contains $\{j_{1},\ldots,j_{p}\}=\{i_{1},\ldots,i_{p}\}$. By using the quadratic relation in  $\dlg_{n,q}$, we conclude that if $j_k = j_{k+1}$,
\begin{align*}
\Pi&=(q-1)\,S_{j_{1}}\cdots S_{j_{k-1}}S_{j_{k}}S_{j_{k+2}}\cdots S_{j_{p}}\,I_{c}+q\,S_{j_{1}}\cdots S_{j_{k-1}}I_{j_{k}}S_{j_{k+2}}\cdots S_{j_{p}}\,I_{c}\\
&=(q-1)\,S_{j_{1}}\cdots S_{j_{k-1}}S_{j_{k}}S_{j_{k+2}}\cdots S_{j_{p}}\,I_{c}+q\,S_{j_{1}}\cdots S_{j_{k-1}}S_{j_{k+2}}\cdots S_{j_{p}}\,I_{c}
\end{align*}
because $I_{j_{k}}I_{c}=I_{c}$. Consequently, by induction on $p$, any product $\Pi$ is a $\Z[q]$-linear combination of products $T_{\tau,c}$ (and with the same composition $c$ for all the terms of the linear combination). So, the reduced products $T_{\sigma,c}$ span linearly $\dlg_{n,q}$ when $(\sigma,c)$ runs over composed permutations of order $n$.
If $c$ is in $\Comp_{n}$, we define a morphism of $\C(q)$-algebras from $\dlg_{n,q}$ to $\IH_{c,q}$ by
$$\psi_{c}(S_{i})=\begin{cases} S_{i} &\text{if }i\text{ is in the code of }c,\\
0 &\text{otherwise,}
\end{cases}
\qquad;\qquad
\psi_{c}(I_{i})=\begin{cases} 1 &\text{if }i\text{ is in the code of }c,\\
0 &\text{otherwise.}
\end{cases}$$
The elements $\psi_{c}(S_{i})$ and $\psi_{c}(I_{i})$ sastify in $\IH_{c,q}$ the relations of the generators $S_{i}$ and $I_{i}$ in $\dlg_{n,q}$. So, there is indeed such a morphism of algebras $\psi_{c}:\dlg_{n,q}\to\IH_{c,q}$, and one has in fact $\psi_{c}(T_{\sigma,b})=T_{\sigma}$ if $\pi(b) \leq \pi(c)$, and $0$ otherwise.
Let us consider the direct sum of algebras $\IH_{\Comp_{n},q}=\bigoplus_{c \in \Comp_{n}} \IH_{c,q}$, and the direct sum of morphisms $\psi=\bigoplus_{c\in \Comp_{n}}\,\psi_{c}$. We denote the basis vectors $[0,0,\ldots,(T_{\sigma}\in \IH_{c,q}),\ldots,0]$ of $\IH_{\Comp_{n},q}$ by $T_{\sigma \in \IH_{c,q}}$; in particular,
$$\psi(T_{\sigma,c})=\sum_{d\geq c} T_{\sigma \in \IH_{d,q}}$$
for any composed permutation $(\sigma,c)$. As a consequence, the map $\psi$ is surjective, because 
$$\psi\left(\sum_{d\geq c}\mu(c,d) \,T_{\sigma,c} \right)=T_{\sigma \in \IH_{c,q}} $$
where $\mu(c,d)=\mu(\pi(c),\pi(d))=(-1)^{\ell(c)-\ell(d)}$ is the M\"obius function of the hypercube lattice of compositions. If $\sigma$ is a permutation, we denote by $\mathrm{orb}(\sigma)$ the set partition whose parts are the orbits of $\sigma$. Since the families $(T_{\sigma,c})_{\mathrm{orb}(\sigma)\leq \pi(c)}$ and $(T_{\sigma \in \IH_{c,q}})_{\mathrm{orb}(\sigma)\leq \pi(c)}$ have the same cardinality $\dim \dlg_{n}$, we conclude that
 $(T_{\sigma,c})_{\mathrm{orb}(\sigma)\leq \pi(c)}$ is a $\C(q)$-linear basis of $\dlg_{n,q}$ and that $\psi$ is an isomorphism of $\C(q)$-algebras.
\end{proof}\bigskip

Notice that the second part of Theorem \ref{amnesia} is the $q$-analog of Corollary 3.2 in \cite{IK99}. To conclude this part, we have to build the inverse limit $\dlg_{\infty,q}=\varprojlim \dlg_{n,q}$, but this is easy thanks to the specializations evoked in the third part of Proposition \ref{psychoactive}. Hence, if $\phi_{N,n} : \dlg_{N,q} \to \dlg_{n,q}$ is the map that sends the generators $I_{i\geq n}$ and $S_{i \geq n}$ to zero and that preserves the other generators, then $(\phi_{N,n})_{N \geq n}$ is a system of compatible maps, and these maps behave well with respect to the filtration $\deg T_{\sigma,c}=|\mathrm{code}(c)|$. Consequently, there is a projective limit $\dlg_{\infty,q}$ whose elements are the infinite linear combinations of $T_{\sigma,c}$, with $\sigma$ finite permutation in $\sym_{\infty}$ and $c$ infinite composition compatible with $\sigma$ and with almost all its parts of size $1$.  \bigskip

\noindent It is not true that two elements $x$ and $y$ in $\dlg_{\infty,q}$ are equal if and only if their projections $\pr_{n}(\phi_{\infty,n}(x))$ and  $\pr_{n}(\phi_{\infty,n}(y))$ are equal for all $n$: for instance,
$$T[21|34|5|6|\cdots]=S_{1}I_{1}I_{3} \quad\text{and}\quad T[2134|5|6|\cdots]=S_{1} I_{1}I_{2}I_{3}$$
have the same projections in all the Hecke algebras (namely, $S_{1}$ if $n\geq4$ and $0$ otherwise), but they are not equal. However, the result is true if we consider only the subalgebras $\dlg_{n,q}'\subset \dlg_{n,q}$ spanned by the $T_{\sigma,c}$ with $c=(k,1,\ldots,1)$ --- then, $\sigma$ may be considered as a partial permutation of $\lle 1,k\rre$.

\begin{proposition}\label{cheshirecat}
For any $n$, the vector space $\dlg_{n,q}'$ spanned by the $T_{\sigma,c}$ with $c=(k,1^{n-k})$ is a subalgebra of $\dlg_{n,q}$. In the inverse limit $\dlg_{\infty,q}'\subset \dlg_{\infty,q}$, the projections $\pr_{\infty,n}=\pr_{n}\circ \phi_{\infty,n}$ separate the vectors:
$$\forall x,y \in \dlg_{\infty,q}',\qquad\big(\forall n,\,\,\pr_{\infty,n}(x)=\pr_{\infty,n}(y)\big) \,\,\iff\,\, \big(x=y\big) .$$
\end{proposition}
\begin{proof}
The supremum of two compositions $(k,1^{n-k})$ and $(l,1^{n-l})$ is $(m,1^{n-m})$ with $m=\max(k,l)$; consequently, $\dlg_{n,q}'$ is indeed a subalgebra of $\dlg_{n,q}$. Any element $x$ of the projective limit $\dlg_{\infty,q}'$ writes uniquely as
$$x=\sum_{k=0}^{\infty}\sum_{\sigma \in \sym_{k}} a_{\sigma,k}(x)\,T_{\sigma,(k,1^{\infty})}.$$
Suppose that $x$ and $y$ have the same projections, and let us fix a permutation $\sigma$. There is a minimal integer $k$ such that $\sigma \in \sym_{k}$, and $a_{\sigma,k}(x)$ is the coefficient of $T_{\sigma}$ in $\pr_{\infty,k}(x)$; consequently, $a_{\sigma,k}(x)=a_{\sigma,k}(y)$. Now, $a_{\sigma,k}(x)+a_{\sigma,k+1}(x)$ is the coefficient of $T_{\sigma}$ in $\pr_{\infty,k+1}(x)$, so one has also $a_{\sigma,k}(x)+a_{\sigma,k+1}(x)=a_{\sigma,k}(y)+a_{\sigma,k+1}(y)$, and $a_{\sigma,k+1}(x)=a_{\sigma,k+1}(y)$. By using the same argument and by induction on $l$, we conclude that $a_{\sigma,k+l}(x)=a_{\sigma,k+l}(y)$ for every $l$, and therefore $x=y$. We have then proved that the projections separate the vectors in $\dlg_{\infty,q}'$.
\end{proof}\bigskip

\section{Bases of the center of the Hecke algebra}
In the following, $\Zen_{n,q}$ is the center of $\IH_{n,q}$. We have already given a characterization of the \textbf{Geck-Rouquier central elements} $\Gamma_{\lambda}$, and they form a linear basis of $\Zen_{n,q}$ when $\lambda$ runs over $\Part_{n}$. Let us write down explicitly this basis when $n=3$:
$$\Gamma_{3}=T_{231}+T_{312}+(q-1)q^{-1}\,T_{321}\qquad;\qquad \Gamma_{2,1}=T_{213}+T_{132}+q^{-1}\,T_{321}\qquad;\qquad \Gamma_{1,1,1}=T_{123}$$
The first significative example of Geck-Rouquier element is actually when $n=4$. Thus, if one considers
\begin{align*}
\Gamma_{3,1}&=T_{1342}+T_{1423}+T_{2314}+T_{3124}+q^{-1}\,(T_{2431}+T_{4132}+T_{3214}+T_{4213})\\
&\quad+(q-1)q^{-1}\,(T_{1432}+T_{3214})+(q-1)q^{-2}\,(T_{3421}+T_{4312}+2\,T_{4231})+(q-1)^{2}q^{-3}\,T_{4321},
\end{align*}
the terms with coefficient $1$ are the four minimal $3$-cycles in $\sym_{4}$; the terms whose coefficients specialize to $1$ when $q=1$ are the eight $3$-cycles in $\sym_{4}$; and the other terms are not minimal in their conjugacy classes, and their coefficients vanish when $q=1$.\bigskip

It is really unclear how one can lift these elements to the Hecke algebras of composed permutations; fortunately, the center $\Zen_{n,q}$ admits other linear bases that are easier to pull back from $\IH_{n,q}$ to $\dlg_{n,q}$. In \cite{Las06}, seven different bases for $\Zen_{n,q}$ are studied\footnote{One can also consult \cite{Jon90} and \cite{Fra99}.}, and it is shown that up to diagonal matrices that depend on $q$ in a polynomial way, the transition matrices between these bases are the same as the transition matrices between the usual bases of the algebra of symmetric functions. We shall only need the \textbf{norm basis} $N_{\lambda}$, whose properties are recalled in Proposition \ref{madhatter}. If $c$ is a composition of $n$ and $\sym_{c}$ is the corresponding Young subgroup of $\sym_{n}$, it is well-known that each coset in $\sym_{n}/\sym_{c}$ or $\sym_{c}\backslash \sym_{n}$ has a unique representative $\omega$ of minimal length which is called the \textbf{distinguished representative} --- this fact is even true for parabolic double cosets. In what follows, we rather work with right cosets, and the distinguished representatives of $\sym_{c} \backslash \sym_{n}$ are precisely the permutation words whose recoils are contained in the set of descents of $c$. So for instance, if $c=(2,3)$, then
$$\sym_{(2,3)}\backslash \sym_{5}=\{12345,13245,13425,13452,31245,31425,31452,34125,34152,34512 \}=12 \sqcup\hspace{-1.32mm}\sqcup\, 345.$$

\begin{proposition}\label{madhatter} \cite[Theorem 7]{Las06}
If $c$ is a composition of $n$, let us denote by $N_{c}$ the element
$$\sum_{\omega \in \sym_{c}\backslash \sym_{n}} q^{-\ell(\omega)}\,T_{\omega^{-1}}\,T_{\omega}$$
in the Hecke algebra $\IH_{n,q}$. Then, $N_{c}$ does not depend on the order of the parts of $c$, and the $N_{\lambda}$ form a linear basis of $\Zen_{n,q}$ when $\lambda$ runs over $\Part_{n}$ --- in particular, the norms $N_{c}$ are central elements. Moreover,
$$(\Gamma_{\lambda})_{\lambda \in \Part_{n}}=D\cdot M2E \cdot(N_{\mu})_{\mu \in \Part_{n}},$$
where $M2E$ is the transition matrice between monomial functions $m_{\lambda}$ and elementary functions $e_{\mu}$, and $D$ is the diagonal matrix with coefficients $\left(q/(q-1)\right)^{n-\ell(\lambda)}$.
\end{proposition}\bigskip

So for instance, $\Gamma_{3}=q^{2}\,(q-1)^{-2}\,(3\,N_{3}-3\,N_{2,1}+N_{1,1,1})$, because $m_{3}=3\,e_{3}-3\,e_{2,1}+e_{1,1,1}$. Let us write down explicitly the norm basis when $n = 3$:
\begin{align*}&N_{3}=T_{123}\qquad;\qquad N_{2,1}=3\,T_{123}+(q-1)q^{-1}\,(T_{213}+T_{132})+(q-1)q^{-2}\,T_{321}\\
&N_{1,1,1}=6\,T_{123}+3(q-1)q^{-1}\,(T_{213}+T_{132})+(q-1)^{2}q^{-2}\,(T_{231}+T_{312})+(q^{3}-1)q^{-3}\,T_{321}
\end{align*}
We shall see hereafter that these norms have natural preimages by the projections $\pr_{n}$ and $\pr_{\infty,n}$.\bigskip

\section{Generic norms and the Hecke-Ivanov-Kerov algebra}
Let us fix some notations. If $c$ is a composition of size $|c|$ less than $n$, then $c \uparrow n$ is the composition $(c_{1},\ldots,c_{r},n-|c|)$, $J_{c}=I_{1}I_{2}\cdots I_{|c|-1}$, and 
$$M_{c,n}=\sum_{\omega \in \sym_{c\uparrow n}\backslash \sym_{n}} q^{-\ell(\omega)}\,T_{\omega^{-1}}\,T_{\omega}\,J_{c},$$
the products $T_{\omega}$ being considered as elements of $\dlg_{n,q}$. So, $M_{c,n}$ is an element of $\dlg_{n,q}$, and we set $M_{c,n}=0$ if $|c|>n$.

\begin{proposition}\label{humanbomb}
For any $N,n$ and any composition $c$, $\phi_{N,n}(M_{c,N})=M_{c,n}$, and $\pr_{n}(M_{c,n})=N_{c \uparrow n}$ if $|c|\leq n$, and $0$ otherwise. On the other hand, $M_{c,n}$ is always in $\dlg_{n,q}'$.
\end{proposition}

\begin{proof}
Because of the description of distinguished representatives of right cosets by positions of recoils, if $|c|\leq n$, then the sum $M_{c,n}$ is over permutation words $\omega$ with recoils in the set of descents of $c$ (notice that we include $|c|$ in the set of descents of $c$). Let us denote by $R_{c,n}$ this set of words, and suppose that $|c|\leq n-1$. If $\omega \in R_{c,n}$ is such that $\omega(n) \neq n$, then $T_{\omega}$ involves $S_{n-1}$, so the image  by $\phi_{n,n-1}$ of the corresponding term in $M_{c,n}$ is zero. On the other hand, if $\omega(n)=n$, then any reduced decomposition of $T_{\omega}$ does not involve $S_{n-1}$, so the corresponding term in $M_{c,n}$ is preserved by $\phi_{n,n-1}$. Consequently, $\phi_{n,n-1}(M_{c,n})$ is a sum with the same terms as $M_{c,n}$, but with $\omega$ running over $R_{c,n-1}$; so, we have proved that $\phi_{n,n-1}(M_{c,n})=M_{c,n-1}$ when $|c|\leq n-1$. The other cases are much easier: thus, if $|c|=n$, then $M_{c,n-1}=0$, and $\phi_{n,n-1}(M_{c,n})$ is also zero because $\phi_{n,n-1}(J_{c})=0$. And if $|c|>n$, then $M_{c,n}$ and $M_{c,n-1}$ are both equal to zero, and again $\phi_{n,n-1}(M_{c,n})=M_{c,n-1}$. Since $$\phi_{N,n}=\phi_{n+1,n}\circ \phi_{n+2,n+1}\circ\cdots\circ\phi_{N,N-1},$$ we have proved the first part of the proposition, and the second part is really obvious. \bigskip

\noindent Now, let us show that $M_{c,n}$ is in $\dlg_{n,q}'$. Notice that the result is trivial if $|c|>n$, and also if $|c|=n$, because we have then $J_{c}=I_{(n)}$, and therefore $d=(n)$ for any composed permutation $(\sigma,d)$ involved in $M_{c,|c|}$. Suppose then that $|c|\leq n-1$. Because of the description of $\sym_{d}\backslash \sym_{|d|}$ as a shuffle product, any distinguished representative $\omega$ of $\sym_{c \uparrow n} \backslash \sym_{n}$ is the shuffle of a distinguished representative $\omega_{c}$ of $\sym_{c}\backslash \sym_{|c|}$ with the word $|c|+1,|c|+2,\ldots,n$. For instance, $5613724$ is the distinguished representative of a right $\sym_{(2,2,3)}$-coset, and it is a shuffle of $567$ with the distinguished representative $1324$ of a right $\sym_{(2,2)}$-coset. Let us denote by $s_{i_{1}}\cdots s_{i_{r}}$ a reduced expression of $\omega_{c}$, and by $j_{|c|+1},\ldots,j_{n}$ the positions of $|c|+1,\ldots,n$ in $\omega$. Then, it is not difficult to see that
$$s_{i_{1}}\cdots s_{i_{r}}\,\times\,(s_{|c|} s_{|c|-1} \cdots s_{j_{|c|+1}})\,(s_{|c|+1}s_{|c|}\cdots s_{j_{|c|+2}})\,\cdots\, (s_{n-1}s_{n-2}\cdots s_{j_{n}})$$
is a reduced expression for $\omega$; for instance, $s_{2}$ is the reduced expression of $1324$, and
$$s_{2}\,\times\,(s_{4}s_{3}s_{2}s_{1})\,(s_{5}s_{4}s_{3}s_{2})\,(s_{6}s_{5})$$
is a reduced expression of $5613724$. From this, we deduce that $T_{\omega}\,J_{c}=T_{\omega,(k,1^{n-k})}$, where $k$ is the highest integer in $\lle |c|+1, n \rre$ such that $j_{k}<k$ --- we take $k=|c|$ if $\omega=\omega_{c}$. Then, the multiplication by $T_{\omega^{-1}}$ cannot fatten the composition anymore, so $T_{\omega^{-1}}T_{\omega}\,J_{c}$ is a linear combination of $T_{\tau,(k,1^{n-k})}$, and we have proved that $M_{n,c}$ is indeed in $\dlg_{n,q}'$.
\end{proof}\bigskip

From the previous proof, it is now clear that if we consider the infinite sum
$M_{c}=\sum \,q^{-\ell(\omega)}\,T_{\omega^{-1}}\,T_{\omega}\,J_{c}$
over permutation words $\omega \in \sym_{\infty}$ with their recoils in the set of descents of $c$, then $M_{c}$ is the unique element of $\dlg_{\infty,q}$ such that $\phi_{\infty,n}(M_{c})=M_{c,n}$ for any positive integer $n$, and also the unique element of $\dlg_{\infty,q}'$ such that $\pr_{\infty,n}(M_{c})=N_{c \uparrow n}$ for any positive integer $n$ (with by convention $N_{c \uparrow n}=0$ if $|c|>n$). In particular, $M_{c}$ does not depend on the order of the parts of $c$, because this is true for the $N_{c\uparrow n}$ and the projections separate the vectors in $\dlg_{\infty,q}'$. Consequently, we shall consider only elements $M_{\lambda}$ labelled by partitions $\lambda$ of arbitrary size, and call them \textbf{generic norms}. For instance:
$$M_{(2),3}=	
T_{12|3} + 2\,T_{123} + (1-q^{-1})\,(T_{132} + T_{213})
+ (q^{-1}-q^{-2})\,T_{321}$$
In what follows, if $i<n$, we denote by $(S_{i})^{-1}$ the element of $\dlg_{n,q}$ equal to:
$$(S_{i})^{-1}=q^{-1}\,S_{i}+\left(q^{-1}-1\right)I_{i}$$
The product $S_{i}\,(S_{i})^{-1}=(S_{i})^{-1}S_{i}$ equals $I_{i}$ in $\dlg_{n,q}$, and by the specialization $\pr_{n}:\dlg_{n,q} \to \IH_{n,q}$, one recovers $S_{i}\,(S_{i})^{-1}=1$ in the Hecke algebra $\IH_{n,q}$.

\begin{theorem}\label{asylum}
The $M_{\lambda}$ span linearly the subalgebra $\clg_{\infty,q}\subset \dlg_{\infty,q}'$ that consists in elements $x \in \dlg_{\infty,q}'$ such that $I_{i}\,x=S_{i}\,x\,(S_{i})^{-1}$ for every $i$. In particular, any product $M_{\lambda}\,*\,M_{\mu}$ is a linear combination of $M_{\nu}$, and moreover, the terms $M_{\nu}$ involved in the product satisfy the inequality $|\nu| \leq |\lambda|+|\mu|$.
\end{theorem}
\begin{proof} If $I_{i}\,x=S_{i}\,x\,(S_{i})^{-1}$ and $I_{i}\,y=S_{i}\,y\,(S_{i})^{-1}$, then 
$$I_{i}\,xy=I_{i}\,x\,I_{i}\,y=S_{i}\,x\,(S_{i})^{-1}S_{i}\,y\,(S_{i})^{-1}=S_{i}\,x\,I_{i}\,y\,(S_{i})^{-1}=S_{i}\,xy\,(S_{i})^{-1},$$ so the elements that ``commute'' with $S_{i}$ in $\dlg_{\infty,q}$ form a subalgebra. As an intersection, $\clg_{\infty,q}$ is also a subalgebra of $\dlg_{\infty,q}$; let us see why it is spanned by the generic norms. If $\dlg_{\infty,q,i}'$ is the subspace of $\dlg_{\infty,q}$ spanned by the $T_{\sigma,c}$ with $c=(k,1^{\infty})\vee(1^{i-1},2,1^{\infty})$, then the projections separate the vectors in this subspace --- this is the same proof as in Proposition \ref{cheshirecat}. For $\lambda \in \Part$, $I_{i}\,M_{\lambda}$ and $S_{i}\,M_{\lambda}\,(S_{i})^{-1}$ belong to $\dlg_{\infty,q,i}'$, and they have the same projections in $\IH_{n,q}$, because $\pr_{\infty,n}(M_{\lambda})$ is a norm and in particular a central element. Consequently, $I_{i}\,M_{\lambda}=S_{i}\,M_{\lambda}\,(S_{i})^{-1}$, and the $M_{\lambda}$ are indeed in $\clg_{\infty,q}$. Now, if we consider an element $x\in \clg_{\infty,q}$, then for $i<n$, $\pr_{n}(x)=S_{i}\,\pr_{n}(x)\,(S_{i})^{-1}$, so $\pr_{n}(x)$ is in $\Zen_{n,q}$ and is a linear combination of norms:
$$\forall n \in \N,\,\,\,\pr_{n}(x)=\sum_{\lambda \in \Part_{n}} a_{\lambda}(x)\,N_{\lambda}$$
Since the same holds for any difference $x-\sum b_{\lambda} \,M_{\lambda}$, we can construct by induction on $n$ an infinite linear combination $S_{\infty}$ of $M_{\lambda}$ that has the same projections as $x$:
\begin{align*} &\pr_{1}(x)=\sum_{|\lambda|=1} b_{\lambda}\,N_{\lambda} \quad\Rightarrow\quad \pr_{1}\bigg(x-\sum_{|\lambda|=1} b_{\lambda}\,M_{\lambda}\bigg)=0,\,\,\,S_{1}=\sum_{|\lambda|=1} b_{\lambda}\,M_{\lambda}\\
&\pr_{2}\left(x-S_{1}\right)=\sum_{|\lambda|=2} b_{\lambda}\,N_{\lambda} \quad\Rightarrow\quad \pr_{1,2}\bigg(x-\sum_{|\lambda|\leq 2} b_{\lambda}\,M_{\lambda}\bigg)=0,\,\,\,S_{2}=\sum_{|\lambda|\leq 2} b_{\lambda}\,M_{\lambda}\\
&\quad\vdots\\
&\pr_{n+1}\left(x-S_{n}\right)=\sum_{|\lambda|=n+1}b_{\lambda}\,N_{\lambda} \quad\Rightarrow\quad S_{n+1}=S_{n}+\sum_{|\lambda|=n+1}b_{\lambda}\,M_{\lambda}=\sum_{|\lambda|\leq n+1}b_{\lambda}\,M_{\lambda}
\end{align*}
Then, $S_{\infty}=\sum_{\lambda \in \Part} b_{\lambda}\,M_{\lambda}$ is in $\dlg_{\infty,q}'$ and has the same projections as $x$, so $S_{\infty}=x$. In particular, since $\clg_{\infty,q}$ is a subalgebra, a product $M_{\lambda}\,*\,M_{\mu}$ is in $\clg_{\infty,q}$ and is an \emph{a priori} infinite linear combination of $M_{\nu}$:
$$\forall \lambda,\mu,\,\,\,M_{\lambda}\,*\,M_{\mu}=\sum \,g_{\lambda\mu}^{\nu} \,M_{\nu}$$
Since the norms $N_{\lambda}$ are defined over $\Z[q,q^{-1}]$, by projection on the Hecke algebras $\IH_{n,q}$, one sees that the $g_{\lambda\mu}^{\nu}$ are also in $\Z[q,q^{-1}]$ --- in fact, they are \emph{symmetric} polynomials in $q$ and $q^{-1}$. It remains to be shown that the previous sum is in fact over partitions $|\nu|$ with $|\nu|\leq |\lambda|+|\mu|$; we shall see why this is true in the last paragraph\footnote{Unfortunately, we did not succeed in proving this result with adequate filtrations on $\dlg_{\infty,q}$ or $\dlg_{\infty,q}'$.}.
\end{proof}\bigskip

For example, $M_{1}\,*\,M_{1}=M_{1}+(q+1+q^{-1})\,M_{1,1}-(q+2+q^{-1})\,M_{2}$, and from this generic identity one deduces the expression of any product $(N_{(1)\uparrow n})^{2}$, \emph{e.g.},
$$N_{1,1}^{\,2}=(q+2+q^{-1})\,(N_{1,1}-N_{2}) \qquad;\qquad N_{3,1}^{\,2}=N_{3,1}+(q+1+q^{-1})\,N_{2,1,1}-(q+2+q^{-1})\,N_{2,2}.$$
Let us denote by $\alg_{\infty,q}$ the subspace of $\clg_{\infty,q}$ whose elements are \emph{finite} linear combinations of generic norms; this is in fact a subalgebra, which we call the \textbf{Hecke-Ivanov-Kerov algebra} since it plays the same role for Iwahori-Hecke algebras as $\alg_{\infty}$ for symmetric group algebras.\bigskip

\section{Completion of partitions and symmetric functions}
The proof of Theorem \ref{holdupmental} and of the last part of Theorem \ref{asylum} relies now on a rather elementary property of the transition matrices $M2E$ and $E2M$. By convention, we set $e_{\lambda \uparrow n}=0$ if $|\lambda|>n$, and $m_{\lambda \rightarrow n}=0$ if $|\lambda|+\ell(\lambda)>n$. Then:
\begin{proposition}\label{psychosocial}
There exists polynomials $P_{\lambda\mu}(n)\in\Q[n]$ and $Q_{\lambda\mu}(n) \in \Q[n]$ such that
$$\forall \lambda,n,\qquad m_{\lambda \rightarrow n} = \sum_{\mu' \leq_{d}\, \lambda} P_{\lambda\mu}(n) \,\,e_{\mu \uparrow n}\quad\mathrm{and}\quad 
                                          e_{\lambda \uparrow n} = \sum_{\mu \leq_{d}\, \lambda'} Q_{\lambda\mu}(n) \,\,m_{\mu \rightarrow n},$$
where $\mu \leq_{d}\, \lambda$ is the domination relation on partitions.
\end{proposition}

\noindent This fact follows from the study of the Kotska matrix elements $K_{\lambda,\mu\rightarrow n}$, see \cite[\S1.6, in particular the example 4. (c)]{Mac95}. It can also be shown directly by expanding $e_{\lambda \uparrow n}$ on a sufficient number of variables and collecting the monomials; this simpler proof explains the appearance of binomial coefficients $\binom{n}{k}$. For instance,
\begin{align*}m_{2,1\rightarrow n}&=e_{2,1 \uparrow n}-3\,e_{3\uparrow n}-(n-3)\,e_{1,1 \uparrow n} +(2n-8)\,e_{2 \uparrow n} +(2n-5)\,e_{1 \uparrow n}- n(n-4)\,e_{\uparrow n},\\
e_{2,1\uparrow n}&=\frac{n(n-1)(n-2)}{2}\,m_{\rightarrow n}+\frac{(n-2)(3n-7)}{2}\,m_{1 \rightarrow n}+(3n-10)\,m_{1,1\rightarrow n}+3\,m_{1,1,1\rightarrow n}\\
&+(n-3)\,m_{2 \rightarrow n}+m_{2,1 \rightarrow n}.
\end{align*}
In the following, $N_{\lambda,n}=N_{\lambda \uparrow n}$ if $|\lambda|\leq n$, and $0$ otherwise. Because of the existence of the projective limits $M_{\lambda}$, we know that  $N_{\lambda,n}\,*\,N_{\mu,n}=\sum_{\nu}\,g_{\lambda\mu}^{\nu}\,N_{\nu,n}$, where the sum is not restricted. But on the other hand, by using Proposition \ref{madhatter} and the second identity in Proposition \ref{psychosocial}, one sees that
$$N_{\lambda,n}\,*\,N_{\mu,n}=\sum_{|\rho|\leq |\lambda|,\,\,|\sigma|\leq|\mu|}h_{\lambda\mu}^{\rho\sigma}(n)\,\,\,\Gamma_{\rho,n}\,*\,\Gamma_{\sigma,n}, \quad\text{with the }h_{\lambda\mu}^{\rho\sigma}(n) \in \Q[n,q,q^{-1}].$$
Because of the result of Francis and Wang, the latter sum may be written as $\sum_{|\tau|\leq |\lambda|+|\mu|}i_{\lambda\mu}^{\tau}(n)\,\Gamma_{\tau,n}$, and by using the first identity of Proposition \ref{psychosocial}, one has finally
$$ N_{\lambda,n}\,*\,N_{\mu,n}=\sum_{|\nu| \leq |\lambda|+|\mu|}\,j_{\lambda\mu}^{\nu}(n)\,N_{\nu,n}, \quad\text{with the }j_{\lambda\mu}^{\nu}(n) \in \Q[n,q,q^{-1}].$$
From this, it can be shown that the first sum $\sum_{\nu}\,g_{\lambda\mu}^{\nu}\,N_{\nu,n}$ is in fact restricted on partitions $|\nu|$ such that $|\nu| \leq |\lambda|+|\mu|$, and because the projections separate the vectors of $\dlg_{\infty,q}'$, this implies that $M_{\lambda}\,*\,M_{\mu}=\sum_{|\nu| \leq |\lambda|+|\mu|} g_{\lambda\mu}^{\nu}\,M_{\nu}$, so the last part of Theorem \ref{asylum} is proved. Finally, by reversing the argument, one sees that the $a_{\lambda\mu}^{\nu}(n,q)$ are in $\Q[n](q)$:
\begin{align*}
\Gamma_{\lambda,n}\,*\,\Gamma_{\mu,n}&=(q/(q-1))^{|\lambda|+|\mu|}\,\sum_{\rho,\sigma}P_{\lambda\rho}(n)\,P_{\mu\sigma}(n)\,\,N_{\rho,n}\,*\,N_{\sigma,n} \\
&=(q/(q-1))^{|\lambda|+|\mu|}\,\sum_{\rho,\sigma,\tau}\,P_{\lambda\rho}(n)\,P_{\mu\sigma}(n)\,g_{\rho\sigma}^{\tau}\,N_{\tau,n} \\
&=\sum_{\rho,\sigma,\tau,\nu}(q/(q-1))^{|\lambda|+|\mu|-|\nu|}\,P_{\lambda\rho}(n)\,P_{\mu\sigma}(n)\,g_{\rho\sigma}^{\tau}(q)\,Q_{\tau\nu}(n)\,\Gamma_{\nu,n}=\sum_{\nu} a_{\lambda\mu}^{\nu}(n,q)\,\Gamma_{\nu,n}
\end{align*}
with $a_{\lambda\mu}^{\nu}(n,q)=(q/(q-1))^{|\lambda|+|\mu|-|\nu|}\,(P^{\otimes 2}(n)\,g(q)\,Q(n))_{\lambda\mu}^{\nu}$ in tensor notation. And since the $\Gamma_{\lambda}$ are known to be defined over $\Z[q,q^{-1}]$, the coefficients $a_{\lambda\mu}^{\nu}(n,q) \in \Q[n](q)$ are in fact\footnote{They are even in $\Q_{\Z}[n]\otimes\Z[q,q^{-1}]$, where $\Q_{\Z}[n]$ is the $\Z$-module of polynomials with rational coefficients and integer values on integers; indeed, the matrices $M2E$ and $E2M$ have integer entries. It is well known that $\Q_{\Z}[n]$ is spanned over $\Z$ by the binomials $\binom{n}{k}$.} in $\Q[n,q,q^{-1}]$. Using this technique, one can for instance show that
$$(\Gamma_{(1),n})^{2}=\frac{n(n-1)}{2}\,q\,\Gamma_{(0),n}+(n-1)\,(q-1)\,\Gamma_{(1),n}+(q+q^{-1})\,\Gamma_{(1,1),n}+(q+1+q^{-1})\,\Gamma_{(2),n},$$
and this is because $m_{1\rightarrow n}=e_{1 \uparrow n}-n\,e_{\uparrow n}$ and $e_{1\uparrow n}=n\,m_{\rightarrow n}+m_{1\rightarrow n}$. Let us conclude by two remarks. First, the reader may have noticed that we did not construct generic conjugacy classes $F_{\lambda} \in \alg_{\infty,q}$ such that $\pr_{\infty,n}(F_{\lambda})=\Gamma_{\lambda,n}$; since the Geck-Rouquier elements themselves are difficult to describe, we had little hope to obtain simple generic versions of these $\Gamma_{\lambda}$. Secondly, the Ivanov-Kerov projective limits of other group algebras --- \emph{e.g.}, the algebras of the finite reductive Lie groups $\mathrm{GL}(n,\mathbb{F}_{q})$, $\mathrm{U}(n,\mathbb{F}_{q^{2}})$, etc. --- have not yet been studied. It seems to be an interesting open question.
\bigskip

\bibliographystyle{alpha}
\bibliography{fpsac}

\end{document}